\documentclass[preprint,review,10pt]{elsarticle}
\usepackage{amsmath,amssymb,amsfonts,amsthm}
\usepackage{hyperref}
\usepackage{graphicx}
\usepackage{mathrsfs}
\usepackage{verbatim}
\usepackage{float}
\usepackage{color}
\usepackage{cases}

\newtheorem{assumption}{Assumption}

\newtheorem{lemma}{Lemma}
\newtheorem{remark}{Remark}
\newtheorem{theorem}{Theorem}

\journal{Applied Numerical Mathematics}

\begin{document}

\begin{frontmatter}
	
	

	\title{Supercloseness of the LDG method for a two-dimensional singularly perturbed convection-diffusion problem on Bakhvalov-type mesh \tnoteref{funding}}
	
	\tnotetext[funding]{
		%Jin Zhang is supported by National Science Foundation of China (11771257), Shandong Provincial Natural Science Foundation, China (ZR2017MA003) and  A Project of Shandong Province Higher Educational Science and Technology Program (J17KA169); Xiaowei Liu is supported by National Science Foundation of China (11601251), Shandong Provincial Natural Science Foundation, China (ZR2016AM13) and  A Project of Shandong Province Higher Educational Science and Technology Program (J16LI10).
		NSFC grants 11771257 and 11601251 support this research.
	}
	
	\author[label1] {Chunxiao Zhang\fnref{cor1}}
	\author[label1] {Jin Zhang \corref{cor2}}
	\author[label1] {Wenchao Zheng \fnref{cor3}}
	\fntext[cor1] {Email: chunxiaozhangang@outlook.com }
	\cortext[cor2] {Corresponding author: jinzhangalex@hotmail.com}
	\fntext[cor3] {Email: Superwenchao@hotmail.com}
	
	\address[label1]{School of Mathematical Sciences, Shandong Normal University, Jinan 250014, Shandong Province, PR China}
	
	\begin{abstract}
		In this paper, we focus on analyzing the supercloseness property of a two-dimensional singularly perturbed convection-diffusion problem with exponential boundary layers. The local discontinuous Galerkin (LDG) method with piecewise tensor-product polynomials of degree $k$ is applied to Bakhvalov-type mesh. By developing special two-dimensional local Gauss-Radau projections and establishing a novel interpolation, supercloseness of an optimal order $k+1$ can be achieved on Bakhvalov-type mesh. It is crucial to highlight that this supercloseness result is independent of the singular perturbation parameter $\varepsilon$.
	\end{abstract}
	
	
\end{frontmatter}

% \linenumbers


\section{Introduction}
Take the following two-dimensional singularly perturbed problem into consideration:
\begin{equation}\label{eq1.1}
	\begin{aligned}
		-\varepsilon \Delta u+\boldsymbol{\alpha} \cdot \nabla u+bu=&f \quad &&\text{in $ \Omega =(0,1)^2$},\\
		u=&0 \quad &&\text{on $ \partial\Omega$},
	\end{aligned}
\end{equation}
where the coefficient $\varepsilon$ is a positive parameter, $\boldsymbol{\alpha}=\boldsymbol{\alpha}(x,y)$, $b=b(x,y)$ and $f=f(x,y)$ are sufficiently smooth and satisfy for $\forall(x,y)\in \bar{\Omega}$, 
$\boldsymbol{\alpha}(x,y)=(a_1(x,y),a_2(x,y))>(0,0),\quad b(x,y)-\frac{1}{2}\nabla \cdot \boldsymbol{\alpha} (x,y)>0
$. Under these assumptions, utilizing the Lax-Milgram lemma allows for a demonstration of the existence of a unique weak solution in $H_0^1(\Omega)\cap H_2(\Omega)$ for equation \eqref{eq1.1}. This solution typically includes two exponential layers of width $\mathcal{O}(\varepsilon\ln(1/\varepsilon))$ at $x=1$ and $y=1$.

Problem \eqref{eq1.1} can be viewed as a model for solving practical applications like the linearised Navier-Stokes equations at high Reynolds number \cite{Roo1Sty2:2008-Robust}. It is widely known that the solution often exhibits a boundary layer, which is characterized by the fact that, although $u$ remains bounded, its derivatives might be significantly large within narrow regions adjacent to the boundary $\partial\Omega$. This situation undermines the stability and the efficiency of traditional numerical methods. Even on adequately refined, layer-adapted meshes, oscillations may persist. Therefore, numerous stabilized and efficient numerical methods have been put forth in order to obtain accurate results (see \cite{Roo1Sty2:2008-Robust, Lin1:2010-Layer, Mil1ORi2:2012-Fitted}).

In this manuscript, we employ the local discontinuous Galerkin (LDG) method \cite{Coc1Shu2:1998-Local}, a type of finite element method known for its excellent stability and high order accuracy \cite{Xu1Shu2:2010-Local}. Owing to these characteristics, the LDG method is particularly suited for addressing singularly perturbed problems like \eqref{eq1.1}. In recent years, there has been a surge of literature related to the application of the LDG method, as exemplified by \cite{Xie1Zhang2:2009-Numerical, Zhu1Zhang2:2014-Uniform, Cheng1Mei2:2021-Analysis, Cheng1Mei2:2022-Local} and their respective references. 

Supercloseness is a crucial convergence property for the LDG method,  which denotes the difference between the LDG solution and a specific interpolation of the exact solution under certain norms. Investigations into supercloseness of the LDG method for singularly perturbed convection-diffusion problems have been conducted in \cite{Xie1Zhang2:2009-Numerical, Xie1Zhang2:2010-Uniform, Zhu1Cel2:2018-Nodal, Cheng1Jiang2:2022-Supercloseness}. Building upon these studies, researchers have derived optimal supercloseness results on a series of layer-adapted meshes, including Shishkin-type and Bakhvalov-type meshes. When considering the two-dimensional case, however, few results have been established \cite{Xie1Tu2:2014-Supercloseness, Cheng1Jiang2:2022-Supercloseness}, especially on Bakhvalov-type mesh. Difficulties arise from the construction of two-dimensional interpolation and the special width of the layer of Bakhvalov-type mesh.

For equation \eqref{eq1.1}, Cheng et al. \cite{Cheng1Jiang2:2022-Supercloseness} applied the LDG method to systematically analyze supercloseness under an energy norm on various types of meshes. They demonstrated an optimal order of $k+1$ on Shishkin-type mesh and Bakhvalov-Shishkin-type mesh. However, on Bakhvalov-type mesh, their supercloseness result was $\mathcal{O}(N^{-(k+1)}\ln(1/\varepsilon)^{1/2})$, affected by the factor $(\ln(1/\varepsilon))^{1/2}$. This influence is caused by the special layer width $\mathcal{O}(\varepsilon\ln(1/\varepsilon))$ of Bakhvalov-type mesh. From a numerical analysis perspective, removing the factor $(\ln(1/\varepsilon))^{1/2}$ and obtaining the optimal convergence results that are independent of $\varepsilon$ on Bakhvalov-type mesh are very interesting and challenging. In the field of the standard finite element method, Roos \cite{Roos1:2006-Error, Roos1Sty2:2015-Some} and Zhang \cite{Zhang1Liu2:2020-Optimal, Zhang1Liu2:2023-Supercloseness} have successively accomplished uniform convergence and supercloseness, unaffected by $\varepsilon$, on Bakhvalov-type mesh by employing diverse techniques. Correspondingly, in the field of the LDG method, it is necessary to prove supercloseness unaffected by $\varepsilon$ on Bakhvalov-type mesh, that is, eliminate the factor $(\ln(1/\varepsilon))^{1/2}$ in the supercloseness result. However, existing interpolation methods or analytical techniques have not been able to effectively tackle this challenge.

To address this research gap, we first design special two-dimensional local Gauss-Radau projections. Subsequently, we construct a novel interpolation to eliminate the intractable element errors. This new interpolation utilizes the characteristics of the special projections and the structure of Bakhvalov-type mesh. Further contributing to the realization of the optimal supercloseness result is the implementation of some improved analytical techniques. In conclusion, supercloseness of an optimal order $k+1$ under the energy norm can be achieved on Bakhvalov-type mesh, independent of the singular perturbation parameter $\varepsilon$. Our article successfully removes the factor $(\ln(1/\varepsilon))^{1/2}$ and proves a parameter-uniform supercloseness result on Bakhvalov-type mesh, indicating an effective analysis to tackle the issues associated with singularly perturbed convection-diffusion problems of the LDG method in 2D.

This paper is organized as follows: In Section \ref{S2}, we present some regularity results of the solution to problem \eqref{eq1.1}, describe Bakhvalov-type mesh, and introduce the LDG method. In Section \ref{S5}, special two-dimensional local Gauss-Radau projections are defined, and a novel interpolation is established. Besides, some preliminary results are also presented in this section. Finally, the supercloseness property under the energy norm is thoroughly demonstrated in Section \ref{S6}.

General constant $C$ is used in this paper, which is positive and unaffected by the singular perturbation $\varepsilon$ and the mesh parameter $N$.

\section{Regularity results, Bakhvalov-type mesh and the LDG method}\label{S2}
Throughout this article, let $k\ge1$ be a fixed positive integer.
\subsection{Regularity results}
\begin{assumption}\label{regularity results}
	For $\forall (x,y)\in \bar{\Omega}$, $u\in C^{k+2,\iota }(\Omega)$ with $0<\iota <1$ can be decomposed as
	$
	u = S + E_{21} +E_{12} + E_{22},
	$
	where $S$ is the smooth part, $E_{21}$, $E_{12}$ are the exponential layer parts, and $E_{22}$ is the corner layer part. In addition, the follwing inequalities hold:
	\begin{align*}
		&|\partial x^m \partial y^n S(x,y)|\le C,\\
		&|\partial x^m \partial y^n E_{21}(x,y)|\le C\varepsilon^{-m}e^{-\frac{\alpha_1 (1-x)}{\varepsilon}},\\
		&|\partial x^m \partial y^n E_{12}(x,y)|\le C\varepsilon^{-n}e^{-\frac{\alpha_2 (1-y)}{\varepsilon}},\\
		&|\partial x^m \partial y^n E_{22}(x,y)|\le C\varepsilon^{-(m+n)}e^{-\frac{\alpha_1 (1-x)+\alpha_2(1-y)}{\varepsilon}},
	\end{align*}
	for all integers $m$, $n$ with $0<m+n\le k+2$.
\end{assumption}

It is suggested to consult \cite{Zhu1Zhang2:2013} for further insight into this assumption. 
\subsection{Bakhvalov-type mesh}
We consider the Bakhvalov-type mesh introduced in \cite{Roos:2006-Error}, whose transition point is denoted as
$
\tau= \min\{1/2,\rho \varepsilon \ln{\frac{1}{\varepsilon}}\},
$
and the mesh generating function is described by
\begin{equation}\label{eq2.3}
	x=\phi (d)=
	\left\{
	\begin{split}
		& -  \rho \varepsilon \ln (-2(1-\varepsilon)d+1 ) \quad && d\in [0,1/2],\\
		&-t (-d+1)+1\quad && d\in (1/2,1].
	\end{split}
	\right.
\end{equation}
Here $t$
is used to guarantee the continuity of $\phi(1/2)$. For $i = 0, 1, \dots , N$,  we define the mesh points $x_i = \phi(i/N)$. 

The two-dimensional layer-adapted mesh $\{(x_i, y_j ), i,j=1,\ldots,N\}$ used in this paper is constructed by the tensor-product of the one-dimensional layer-adapted meshes in the horizontal and vertical directions, where $x_i$ is defined as \eqref{eq2.3} and $y_j$ is defined similarly, see Figure \ref{fig:paper4-figure}. For $i,j=1,\ldots,N$, set $I_i=(x_{i-1},x_i)$, $J_j=(y_{j-1},y_j)$ and $h_{x,i}=x_i-x_{i-1}$, $h_{y,j}=y_j-y_{j-1}$. Let 
$
\Omega_N=\{K_{ij}: K_{ij} = I_i\times J_j \quad i,j=1,\ldots,N\}
$
denote a rectangle partition of $\Omega$. And $K\in\Omega_N$ denotes a general mesh rectangle.

\begin{figure}[H]
	\centering
	\includegraphics[width=0.7\linewidth]{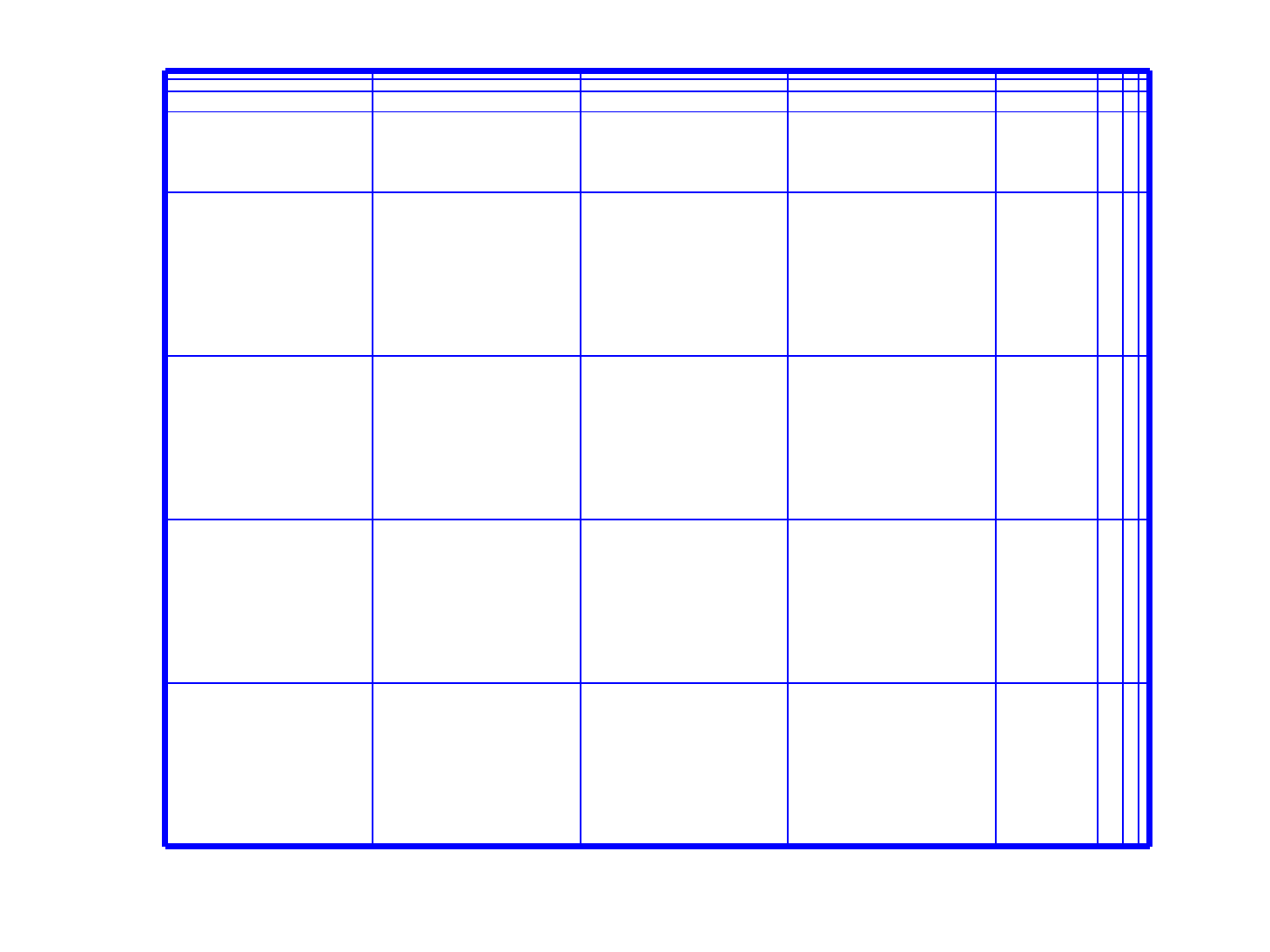}
	\caption{Bakhvalov-type mesh}
	\label{fig:paper4-figure}
\end{figure}

\begin{assumption}
In this paper, we assume that
\begin{enumerate}[(i)]
\item  $\rho\ge k+2$.
\item  $\varepsilon\le N^{-1}$.
\end{enumerate}
\end{assumption}

Then we will introduce some important properties of Bakhvalov-type mesh \eqref{eq2.3}.
\begin{lemma}\label{step}
	\begin{align*}
		&C_1\varepsilon N^{-1}\le h_{x,N}\le C_2\varepsilon N^{-1},\\
		&h_{x,N}\le h_{x,N-1}\le\dots\le h_{x,N/2+2},\\
		&1/4\rho\varepsilon\le h_{x,N/2+2}\le \rho\varepsilon,\\
		&1/2\rho\varepsilon\le h_{x,N/2+1}\le 2\rho N^{-1},\\
		&N^{-1}\le h_{x,i}\le 2N^{-1}\quad 1\le i\le N/2,\\
		&C_3\rho\varepsilon\ln N\le x_{N/2+1}\le C_4\rho\varepsilon\ln N,\quad x_{N/2}\ge C\varepsilon|\ln\varepsilon|.
	\end{align*}
	Bounds for $h_{y,j}$, $1\le j\le N$, are analogous.
\end{lemma}

\subsection{the LDG method}
Define the discontinuous finite element space:
$$\mathcal{V}_N=\{v\in L^2(\Omega):\text{$v|_{K}\in \mathbb{Q}^k(K)$ }, K\in \Omega_N \},$$ where $\mathbb{Q}^k(K)$ represents the space of tensor-product polynomials of degree $k$ in each variable on $K$. Note that the functions within the space $\mathcal{V}_N$ are permitted to be discontinuous at the interfaces of elements.

For $v \in \mathcal{V}_N$, $x\in I_i$ and $y \in J_j$, we use $v^{\pm}_{i,y}=\lim_{x\to x^{\pm}_i}v(x,y)$ and $v^{\pm}_{x,j}=\lim_{y\to y^{\pm}_j}v(x,y)$ to express the traces evaluated from the four directions. Besides, write the jumps on the vertical and horizontal edges as:
\begin{align*}
	&&&[[v]]_{i,y}=v^{+}_{i,y}-v^{-}_{i,y}\text{ for $i=1,2,\ldots,N-1$},\; &&[[v]]_{0,y}=v^{+}_{0,y},\;&[[v]]_{N,y}=-v^{-}_{N,y};\\ 
	&&&[[v]]_{x,j}=v^{+}_{x,j}-v^{-}_{x,j}\text{ for $j=1,2,\ldots,N-1$},\; &&[[v]]_{x,0}=v^{+}_{x,0},\;&[[v]]_{x,N}=-v^{-}_{x,N}.
\end{align*}

Next, we present the LDG method for problem \eqref{eq1.1}. Rewrite \eqref{eq1.1} by an equivalent first-order system:
\begin{equation}\label{eq2.2}
	\left\{
	\begin{aligned}
		&\varepsilon^{-1}p=u_x&&\quad \text{in $\Omega$}, \\
		&\varepsilon^{-1}q=u_y&&\quad \text{in $\Omega$}, \\
		&-p_x-q_y+a_1 u_x+a_2 u_y+bu=f&&\quad \text{in $\Omega$}.
	\end{aligned}
	\right.
\end{equation}
For $y\in J_j$,  $j=1,2,\ldots,N$, define the numerical fluxes by
\begin{equation*}
	\begin{aligned}
		&\hat{P}_{i,y}=
		\left\{
		\begin{aligned}
			& P^{+}_{0,y}+\lambda_1U^{+}_{0,y} \quad &&\text{$i=0$},\\
			& P^{+}_{i,y} \quad &&\text{$i=1,2,\ldots,N-1$},\\
			& P^{-}_{N,y}-\lambda_2U^{-}_{N,y} \quad &&\text{$i=N$},
		\end{aligned}
		\right. \label{eq:Bakhvalov mesh-Roos}\\
		&\hat{U}_{i,y}=
		\left\{
		\begin{aligned}
			& 0 \quad &&\text{$i=0,N$},\\
			& U^{-}_{i,y} \quad &&\text{$i=1,2,\ldots,N-1$},\\
		\end{aligned}
		\right.\\
	\end{aligned}
\end{equation*}
where  $0 \le \lambda_1,\lambda_2 \le C$. When $x\in I_i$, $i=1,2,\ldots,N$, we can define $\hat{Q}_{x,j}$ and $\hat{U}_{x,j}$ for $j=1,2,\ldots,N$ analogously. Let $(\cdot,\cdot)_{D}$ be the inner product in $L^2(D)$ with $D\subseteq \mathbb{R}^2$, and $\left< \cdot,\cdot\right>_{I}$ be the inner product in $L^2(I)$ with $I\subseteq \mathbb{R}$. $Z=(v,s,r)\in \mathcal{V}_N \times \mathcal{V}_N\times \mathcal{V}_N$ $(=:\mathcal{V}_N^3)$ denotes any test function. Then the compact form of the LDG method can be described as: 

Find the LDG solution $W=(U,P,Q)\in \mathcal{V}_N^3$ such that
\begin{equation*}
	B(W;Z)=(f,v)\quad \forall Z=(v,s,r)\in \mathcal{V}_N^3,
\end{equation*}
where
\begin{equation*}
B(W;Z):=\mathcal{B}_1(W;Z)+\mathcal{B}_2(W;Z)+\mathcal{B}_3(W;Z)+\mathcal{B}_4(U;v),
\end{equation*}
and
\begin{equation}\label{LDG2:2}
	\begin{aligned}
&\begin{aligned}
			\mathcal{B}_1(W;Z)=((b-\nabla \cdot \alpha) U,v)+\varepsilon^{-1}(P,s)+\varepsilon^{-1}(Q,r),
		\end{aligned}\\
&\begin{aligned}
		\mathcal{B}_2(W;Z)=(U,s_x)+\sum_{i=1}^{N-1}\sum_{j=1}^{N}\left<U^{-}_{i,y},[[s]]_{i,y}\right>_{J_j}
		+(U,r_y)+\sum_{i=1}^{N}\sum_{j=1}^{N-1}\left<U^{-}_{x,j},[[r]]_{x,j}\right>_{I_i},
	\end{aligned}\\
&\begin{aligned}
\mathcal{B}_3(W;Z)=&(P,v_x)+\sum_{j=1}^{N}\left\{ \sum_{i=0}^{N-1}\left<P^{+}_{i,y},[[v]]_{i,y}\right>_{J_j}-\left<P^{-}_{N,y},v_{N,y}^-\right>_{J_j} \right\}\\
&+(Q,v_y)+\sum_{i=1}^{N}\left\{ \sum_{j=0}^{N-1}\left<Q^{+}_{x,j},[[v]]_{x,j}\right>_{I_i}-\left<Q^{-}_{x,N},v_{x,N}^-\right>_{I_i} \right\},
\end{aligned}\\
&\begin{aligned}
\mathcal{B}_4(U;v)=&-(a_1 U,v_x)-\sum_{j=1}^{N}\left\{\sum_{i=1}^N\left<(a_1)_{i,y} U^{-}_{i,y},[[v]]_{i,y}\right>_{J_j}-\left< \lambda_1 U^{-}_{N,y},v^{-}_{N,y} \right>_{J_j}\right\}\\
&-(a_2 U,v_y)-\sum_{i=1}^{N}\left\{\sum_{j=1}^N\left<(a_2)_{x,j} U^{-}_{x,j},[[v]]_{x,j}\right>_{I_i}-\left< \lambda_2 U^{-}_{x,N},v^{-}_{x,N} \right>_{I_i}\right\}.
\end{aligned}
	\end{aligned}
\end{equation}

Let $w=(u,p,q)$ be the exact solution of \eqref{eq2.2}.
According to \eqref{LDG2:2}, we can define the energy norm as
\begin{equation}\label{energy norm}
	\begin{aligned}
	|||w|||_E^2=B(w;w)=& |||w|||_2^2+\sum_{j=1}^{N}\left\{\sum_{i=0}^{N-1}\frac{1}{2}\left<(a_1)_{i,y},[[u]]^2_{i,y}\right>_{J_j}+\left< \frac{1}{2}(a_1)_{N,y}+\lambda_1, [[u]]^2_{N,y} \right>_{J_j}\right\}\\ 
	&+\sum_{i=1}^{N}\left\{\sum_{j=0}^{N-1}\frac{1}{2}\left<(a_2)_{x,j},[[u]]^2_{x,j}\right>_{I_i}+\left< \frac{1}{2}(a_2)_{x,N}+\lambda_2, [[u]]^2_{x,N} \right>_{I_i}\right\},
\end{aligned}
\end{equation}
where $|||w|||_2^2=:\varepsilon^{-1}||p||^2+\varepsilon^{-1}||q||^2+||(b-\frac{1}{2}\nabla\cdot\boldsymbol{\alpha})^{\frac{1}{2}}u||^2$.

\section{Local Gauss-Radau projections, new interpolation and some preliminary results}\label{S5}
\subsection{Local Gauss-Radau projections}
In this part, we will design several two-dimensional local Gauss-Radau projections. For each $K_{ij}\in\Omega_N$, set $\mathcal{Q}^k(K_{ij}):=\mathcal{P}^k(I_i)\otimes\mathcal{P}^k(J_j)$, where $\mathcal{P}^k(I_i)$ represents the set of polynomials with degree $k$ defined on any one-dimensional interval $I_i$. $\forall z\in C(\overline{K}_{ij})$, we define the local Gauss-Radau projections $\pi-$, $\pi^{-}_x$, $\pi^{-}_y$, $\pi^{+}_x$, $\pi^{+}_y \in\mathcal{Q}^k(K_{ij})$ with $i,j=1,\cdots,N$ as:
\begin{equation*}
	\begin{aligned}
		&\left\{
		\begin{aligned}
			& (\pi^{-}z,v)_{K_{ij}}=(z,v)_{K_{ij}} \quad &&\text{$\forall v\in \mathcal{Q}^{k-1}(K_{ij})$},\\
			& \left<(\pi^{-} z)^{-}_{i,y},v\right>_{J_j}=\left<z^{-}_{i,y},v\right>_{J_j} \quad &&\forall v\in \mathcal{P}^{k-1}(J_j),\\
			& \left<(\pi^{-} z)^{-}_{x,j},v\right>_{I_i}=\left<z^{-}_{x,j},v\right>_{I_i} \quad &&\forall v\in \mathcal{P}^{k-1}(I_i),\\
			&(\pi^-z)(x_i^-,y_j^-)=z(x_i^-,y_j^-),
		\end{aligned}
		\right. \\
		&\left\{
		\begin{aligned}
			& (\pi^{-}_x z,v)_{K_{ij}}=(z,v)_{K_{ij}} \quad &&\text{$\forall v\in \mathcal{P}^{k-1}(I_i)\otimes\mathcal{P}^k(J_j)$},\\
			& \left<(\pi^{-}_x z)^{-}_{i,y},v\right>_{J_j}=\left<z^{-}_{i,y},v\right>_{J_j}\quad &&\forall v\in \mathcal{P}^{k}(J_j),
		\end{aligned}
		\right. \\
		&\left\{
		\begin{aligned}
			& (\pi^{-}_y z,v)_{K_{ij}}=(z,v)_{K_{ij}} \quad &&\text{$\forall v\in \mathcal{P}^{k}(I_i)\otimes\mathcal{P}^{k-1}(J_j)$} ,\\
			& \left<(\pi^{-}_y z)^{-}_{x,j},v\right>_{I_i}=\left<z^{-}_{x,j},v\right>_{I_i} \quad &&\forall v\in \mathcal{P}^{k}(I_i),
		\end{aligned}
		\right. 
	\end{aligned}
\end{equation*}
\begin{equation*}
	\begin{aligned}
		&\left\{
	\begin{aligned}
		& (\pi^{+}_x z,v)_{K_{ij}}=(z,v)_{K_{ij}} \quad &&\text{$\forall v\in \mathcal{P}^{k-1}(I_i)\otimes\mathcal{P}^k(J_j)$} ,\\
		& \left<(\pi^{+}_x z)^{+}_{i,y},v\right>_{J_j}=\left<z^{+}_{i,y},v\right>_{J_j} \quad &&\forall v\in \mathcal{P}^{k}(J_j),
	\end{aligned}
	\right. \\ 
	&\left\{
	\begin{aligned}
		& (\pi^{+}_y z,v)_{K_{ij}}=(z,v)_{K_{ij}} \quad &&\text{$\forall v\in \mathcal{P}^{k}(I_i)\otimes\mathcal{P}^{k-1}(J_j)$},\\
		& \left<(\pi^{+}_y z)^{+}_{x,j},v\right>_{I_i}=\left<z^{+}_{x,j},v\right>_{I_i} \quad &&\forall v\in \mathcal{P}^{k}(I_i).
	\end{aligned}
	\right. \\ 
	\end{aligned}
\end{equation*}

It can be deduced that $\pi^-, \pi^{-}_x, \pi^{-}_y,  \pi^{+}_x, \pi^{+}_y$ exist and are unique under these conditions \cite{Cas1Coc2:2002-Optimal}. Moreover, the projections $\Pi\in\left\lbrace\pi^-, \pi_x^-, \pi_y^-,\pi_x^+,\pi_y^+\right\rbrace$ possess the following stability properties and approximation properties as stated in \cite[Lemma 2.14]{Ape1:1999-Anisotropic}:
\begin{align*}
&\Vert\Pi z\Vert_{L^\infty(K_{ij})}\le C\Vert z\Vert_{L^\infty(K_{ij})},\\
&\Vert z-\Pi z\Vert_{K_{ij}}\le C\left[h_{x,i}^{k+1}\Vert\partial_x^{k+1}z\Vert_{K_{ij}}+h_{y,j}^{k+1}\Vert\partial_y^{k+1}z\Vert_{K_{ij}}\right].
\end{align*}

\subsection{New interpolation}
Then, establish the new interpolation $\pi w=(\mathcal{P}^-u, \mathcal{P}_x^+p, \mathcal{P}_y^+q)$ of the true solution $w=(u,p,q)$, where $\mathcal{P}^-u, \mathcal{P}_x^+p$ and $\mathcal{P}_y^+q$ are defined as:
\begin{equation}\label{new interpolation}
	\begin{aligned}
		&\mathcal{P}^{-} u=
		\left\{
		\begin{aligned}
			& \pi^{-}_x u\quad(x,y)\in [x_{N/2-1}, x_{N/2}]\times([0,y_{N/2-1}]\cup[y_{N/2},1]),\\
			& \pi^{-}_y u\quad(x,y)\in [y_{N/2-1}, y_{N/2}]\times([0,x_{N/2-1}]\cup[x_{N/2},1]),\\
			&\pi^-u\quad(x,y)\in {\text otherwise},
		\end{aligned}
		\right. \\
		&\mathcal{P}^{+}_x  p=
		\begin{aligned}
			\pi^{+}_x p\quad(x,y)\in\Omega,  \\
		\end{aligned}  \\
		&\mathcal{P}^{+}_y  q=
		\begin{aligned}
			\pi^{+}_y q\quad(x,y)\in\Omega.\\
		\end{aligned}
	\end{aligned}
\end{equation}

\begin{remark}
The supercloseness result estimated by the previously defined local Gauss-Radau projections $\pi^-$, $\pi_x^+$, $\pi_y^+$ and interpolation $(\pi^-u, \pi_x^+p, \pi_y^+q)$ (see \cite{Cheng1Mei2:2022-Local} and \cite{Cheng1Jiang2:2022-Supercloseness}) was affected by the factor $(\ln(1/\varepsilon))^{1/2}$. Here, we newly introduce two special local Gauss-Radau projections $\pi_x^-$ and $\pi_y^-$, and construct a new interpolation $(\mathcal{P}^-u, \pi_x^+p, \pi_y^+q)$ to further eliminate some intractable element errors. Utilizing this new interpolation, we succeed in eliminating the factor $(\ln(1/\varepsilon))^{1/2}$ and demonstrating parameter-uniform supercloseness of an optimal order $k+1$ in the following analysis.
\end{remark}
\subsection{Preliminary results}
In this part, we will introduce some preliminary conclusions. 

Define
\begin{equation*}
	\begin{aligned}
		&\eta=w-\pi w=(\eta_u,\eta_p,\eta_q)=(u-P^{-} u,p-P^{+}_x p, q-P^{+}_y q),\\
		&\xi=W-\pi w=(\xi_u,\xi_p,\xi_q)=(U-P^{-} u,P-P^{+}_x p, Q-P^{+}_y q),
 	\end{aligned}
\end{equation*}
where $w$ is the exact solution of \eqref{eq2.2}, $\pi w$ is the interpolation of $w$, and $W$ is the LDG solution. Then, the error $e=(u-U,p-P,q-Q)$ can be represented as
\begin{equation*}
	e=w-W=\eta-\xi.
\end{equation*}

In the following lemma, we will showcase some interpolation error estimates. Note that $\Vert\cdot\Vert_{L^{n}(D)}$ represents norms in the Lebesgue space. When $n=2$, $\Vert \cdot\Vert_{L^{2}(\Omega)}$ is written as $\Vert \cdot\Vert$ for simplicity.
\begin{lemma}\label{interpolation error estimates}
For $i,j=1,\dots,N$, there exists $C$ such that
\begin{align*}
&\Vert\eta_u\Vert_{L^\infty(\Omega)}\le CN^{-(k+1)},\\
&\Vert(\eta_u)_{i,y}^-\Vert_{L^\infty(K_{ij})}+\Vert(\eta_u)_{x,j}^-\Vert_{L^\infty(K_{ij})}\le CN^{-(k+1)},\\
&\Vert(\eta_u)_{i,y}^-\Vert_{J_j}+\Vert(\eta_u)_{x,j}^-\Vert_{I_i}\le CN^{-(k+3/2)},\\
&\sum_{j=1}^N\Vert(\eta_u)_{i,y}^-\Vert_{J_j}^2+\sum_{i=1}^N\Vert(\eta_u)_{x,j}^-\Vert_{I_i}^2\le CN^{-2(k+1)},\\
&\varepsilon^{-1/2}\Vert\eta_p\Vert+\varepsilon^{-1/2}\Vert\eta_q\Vert\le CN^{-(k+1)}.
\end{align*}
\end{lemma}
\begin{proof}
The aforementioned estimates can be analogously derived by consulting \cite[Lemma 4.1]{Cheng1Mei2:2022-Local}, which provides comparable methods and results. 
\end{proof}

Next, we will present a lemma proposed in \cite[Lemma 3.3]{Cheng1Jiang2:2022-Supercloseness} to bound $\xi$.
\begin{lemma}\label{a bound}
There exists a constant $C$ such that
\begin{align*}
&\Vert(\xi_u)_x\Vert_{[x_{N/2},1]\cup[0,1]}+\left(\sum_{i=N/2+1}^N\sum_{j=1}^N h_{x,i}^{-1}\Vert[[\xi_u]]_{i-1,y}\Vert_{J_j}^2\right)^{1/2}\le C\varepsilon^{-1/2}\left[|||\xi|||_E+N^{-(k+1)}\right],\\
&\Vert(\xi_u)_x\Vert_{[0,1]\cup[y_{N/2},1]}+\left(\sum_{i=1}^N\sum_{j=N/2+1}^N h_{y,j}^{-1}\Vert[[\xi_u]]_{x,j-1}\Vert_{I_i}^2\right)^{1/2}\le C\varepsilon^{-1/2}\left[|||\xi|||_E+N^{-(k+1)}\right].
\end{align*}
\end{lemma}

\section{Supercloseness analysis}\label{S6}
Owing to the consistency of numerical fluxes and the smooth assumptions of $u$, $p$ and $q$, one can obtain
\begin{equation}\label{norm2:3}
	B(\xi;Z)=B(\eta;Z) \quad \forall Z\in \mathcal{V}_N^3.
\end{equation}
Set $Z=\xi$ in \eqref{norm2:3}, in light of the definition of $|||\cdot|||_E$ and the Galerkin orthogonality, we can derive
\begin{equation}\label{norm2:4}
	|||\xi|||_E^2=B(\xi;\xi)=B(\eta;\xi):=\sum_{i=1}^{3}\mathcal{B}_i(\eta;\xi)+\mathcal{B}_4(\eta_u;\xi_u),
\end{equation}
here $\mathcal{B}_i (i=1,2,3,4)$ can be expressed in accordance with \eqref{LDG2:2}.

Refer to \cite[Theorem 4.1]{Cheng1Jiang2:2022-Supercloseness}, we can further rewrite the $x$-direction component of $\mathcal{B}_4(\eta_u;\xi_u)$ into the following form:
\begin{equation*}
\begin{aligned}
		\mathcal{B}_4^x(\eta_u;\xi_u)=&-(a_1 \eta_u,(\xi_u)_x)-\sum_{j=1}^{N}\left\{\sum_{i=1}^N\left<(a_1)_{i,y} (\eta_u)^{-}_{i,y},[[\xi_u]]_{i,y}\right>_{J_j}-\left< \lambda_1 (\eta_u)^{-}_{N,y},(\xi_u)^{-}_{N,y} \right>_{J_j}\right\}\\
		:=&\sum_{i=1}^5\mathcal{B}_{4i}^x(\eta_u;\xi_u),
\end{aligned}
\end{equation*}
where
\begin{equation*}
	\begin{aligned}
		&\begin{aligned}
			\mathcal{B}_{41}^x(\eta_u;\xi_u)=-\sum_{i=1}^{N/2}\sum_{j=1}^{N}(a_1)_{ij}\mathcal{D}_{ij}^1(\eta_u,\xi_u),
		\end{aligned}\\
		&\begin{aligned}
			\mathcal{B}_{42}^x(\eta_u;\xi_u)=&-\sum_{i=1}^{N/2}\sum_{j=1}^{N}\left[\left<(a_1-(a_1)_{ij})\eta_u,(\xi_u)_x\right>_{K_{ij}}+\left<((a_1)_{i-1,y}-(a_1)_{ij})(\eta_u)_{i-1,y}^-,(\xi_u)_{i-1,y}^+\right>_{J_j}\right.\\
			&\left.-\left<((a_1)_{i,y}-(a_1)_{ij})(\eta_u)_{i,y}^-,(\xi_u)_{i,y}^-\right>_{J_j}\right],
		\end{aligned}\\
		&\begin{aligned}
			\mathcal{B}_{43}^x(\eta_u;\xi_u)=-\sum_{i=N/2+1}^{N}\sum_{j=1}^{N}\left<a_1\eta_u,(\xi_u)_x\right>_{K_{ij}}-\sum_{i=N/2+1}^{N}\sum_{j=1}^{N}\left<(a_1)_{i,y}(\eta_u)_{i,y}^-,[[\xi_u]]_{i,y}\right>_{J_j},
		\end{aligned}\\
		&\begin{aligned}
			\mathcal{B}_{44}^x(\eta_u;\xi_u)=-\sum_{j=1}^N(a_1)_{N/2,y}\left<(\eta_u)_{N/2,y}^-,(\xi_u)_{N/2,y}^+\right>_{J_j},
		\end{aligned}\\
		&\begin{aligned}
			\mathcal{B}_{45}^x(\eta_u;\xi_u)=\sum_{j=1}^N\lambda_1\left<(\eta_u)_{N,y}^-,(\xi_u)_{N,y}^-\right>_{J_j}.
		\end{aligned}
	\end{aligned}
\end{equation*}
Here $(a_1)_{ij}$ denotes $a_1(x_i,y_j)$ and $
\mathcal{D}_{ij}^1(\eta_u,\xi_u):=\left<\eta_u,(\xi_u)_x\right>_{K_{ij}}-\left<(\eta_u)_{i,y}^-,(\xi_u)_{i,y}^-\right>_{J_j}+\left<(\eta_u)_{i-1,y}^-,(\xi_u)_{i-1,y}^+\right>_{J_j}$. In the following analysis, we will just analyze $\mathcal{B}_{4}^x(\eta_u;\xi_u)$, for $\mathcal{B}_{4}^y(\eta_u;\xi_u)$ (the $y$-direction component of $\mathcal{B}_4(\eta_u;\xi_u)$) can be analyzed similarly.

The estimate approach for $\mathcal{B}_i(\eta;\xi)\ i=1,2,3$ is akin to the one in \cite[Theorem 4.1]{Cheng1Mei2:2022-Local}, and the proof for $\mathcal{B}_{4i}^x(\eta_u;\xi_u)\ i=1,2,5$ closely resembles that in \cite[Theorem 4.1]{Cheng1Jiang2:2022-Supercloseness}. Consequently, we can easily derive that
\begin{equation}\label{have been given-1}
	|\mathcal{B}_i(\eta;\xi)|\le CN^{-(k+1)}|||\xi|||_E\quad i=1,2,3,
\end{equation}
and 
\begin{equation}\label{main-1}
|\mathcal{B}_{4i}^x(\eta_u;\xi_u)|\le CN^{-(k+1)}|||\xi|||_E\quad i=1,2,5.
\end{equation}

For $\mathcal{B}_{43}^x(\eta_u;\xi_u)$, its supercloseness result was $\mathcal{O}(N^{-(k+1)}\ln(1/\varepsilon)^{1/2})$ in the analysis of \cite{Cheng1Jiang2:2022-Supercloseness}. To eliminate $(\ln(1/\varepsilon))^{1/2}$, we attempt to improve the analysis technique proposed in \cite[Lemma 3.3]{Cheng1Jiang2:2022-Supercloseness} by considering the situation on the interval $[x_{N/2},x_{N/2+1}]$ and $[x_{N/2+1},x_N]$ separately, which can effectively avoid the influence of the layer width $\mathcal{O}(\varepsilon\ln(1/\varepsilon))$. As a result of this enhancement, we successfully attain the optimal convergence rate independent of $\varepsilon$, see Lemma \ref{estimate-B43}.

And for $\mathcal{B}_{44}^x(\eta_u;\xi_u)$, previous methods were unable to avoid $(\ln(1/\varepsilon))^{1/2}$. In this article, we newly introduce two special local Gauss-Radau projections $\pi_x^-$, $\pi_y^-$, and establish a new interpolation $(\mathcal{P}^-u,\mathcal{P}_x^+p,\mathcal{P}_u^+q)$ to further eliminate element boundary errors on the region $[x_{N/2-1},x_{N/2}]\times([0,y_{N/2-1}]\cup[y_{N/2},1])$. Then concentrating the analysis on a specific mesh $K_{N/2,N/2}$, where an optimal supercloseness result independent of $\varepsilon$ can be achieved through the utilization of a skillful technique, as demonstrated in Lemma \ref{estimate-B44}.

\begin{lemma}\label{estimate-B43}
\begin{equation*}
|\mathcal{B}_{43}^x(\eta_u;\xi_u)|\le CN^{-(k+1)}\ln^{1/2}N(|||\xi|||_E+N^{-(k+1)}).
\end{equation*}
\end{lemma}
\begin{proof}
\begin{equation}\label{B43}
	\begin{aligned}
\mathcal{B}_{43}^x(\eta_u;\xi_u)=&-\sum_{j=1}^{N}\left<a_1\eta_u,(\xi_u)_x\right>_{K_{N/2+1,j}}-\sum_{j=1}^{N}\left<(a_1)_{N/2+1,y}(\eta_u)_{N/2+1,y}^-,[[\xi_u]]_{N/2+1,y}\right>_{J_j}\\
&-\sum_{i=N/2+2}^{N}\sum_{j=1}^{N}\left<a_1\eta_u,(\xi_u)_x\right>_{K_{ij}}-\sum_{i=N/2+2}^{N}\sum_{j=1}^{N}\left<(a_1)_{i,y}(\eta_u)_{i,y}^-,[[\xi_u]]_{i,y}\right>_{J_j}\\
=&\uppercase\expandafter{\romannumeral1}+\uppercase\expandafter{\romannumeral2}+\uppercase\expandafter{\romannumeral3}+\uppercase\expandafter{\romannumeral4}.
\end{aligned}
\end{equation}

Recall H\"{o}lder inequalities, Lemma \ref{interpolation error estimates} and inverse inequality, we have
\begin{equation}\label{B43-1}
\begin{aligned}
	|\uppercase\expandafter{\romannumeral1}|\le&C\sum_{j=1}^N\Vert(\eta_u)\Vert_{K_{N/2+1,j}}\Vert(\xi_u)_x\Vert_{K_{N/2+1,j}}\\
	\le&C\sum_{j=1}^N N^{-(k+1)}h_{x,N/2+1}^{1/2}h_{y,j}^{1/2}h_{x,N/2+1}^{-1/2}\Vert\xi_u\Vert_{K_{N/2+1,j}}\\
	\le&CN^{-(k+1)}|||\xi|||_E.
\end{aligned}
\end{equation}

H\"{o}lder inequalities, Lemma \ref{interpolation error estimates} and \eqref{energy norm} give
\begin{equation}\label{B43-2}
\begin{aligned}
	|\uppercase\expandafter{\romannumeral2}|\le(\sum_{j=1}^N\Vert(\eta_u)_{N/2+1,y}^-\Vert_{J_j}^2)^{1/2}(\sum_{j=1}^N\Vert[[\xi_u]]_{N/2+1,y}\Vert_{J_j}^2)^{1/2}\le CN^{-(k+1)}|||\xi|||_E.
\end{aligned}
\end{equation}

H\"{o}lder inequalities, \eqref{energy norm}, Lemmas \ref{step}, \ref{interpolation error estimates} and \ref{a bound} yield
\begin{equation}\label{B43-3,4}
\begin{aligned}
|\uppercase\expandafter{\romannumeral3}+\uppercase\expandafter{\romannumeral4}|\le&C\Vert\eta_u\Vert_{[x_{N/2+1},1]\cup[0,1]}\Vert(\xi_u)_x\Vert_{[x_{N/2+1},1]\cup[0,1]}\\
&+C\left(\sum_{i=N/2+2}^{N-1}\sum_{j=1}^N h_{x,i+1}\Vert(\eta_u)_{i,y}^-\Vert_{J_j}^2\right)^{1/2}\left(\sum_{i=N/2+2}^{N-1}\sum_{j=1}^N h_{x,i+1}^{-1}\Vert[[\xi_u]]_{i,y}\Vert_{J_j}^2\right)^{1/2}\\
&+C\left(\sum_{j=1}^N\Vert(\eta_u)_{N,y}^-\Vert_{J_j}^2\right)^{1/2}\left(\sum_{j=1}^N\Vert[[\xi_u]]_{N,y}\Vert_{J_j}^2\right)^{1/2}\\
\le&C\left(\Vert\eta_u\Vert_{[x_{N/2+1},1]\cup[0,1]}^2+\sum_{i=N/2+2}^{N-1}\sum_{j=1}^N h_{x,i+1}\Vert(\eta_u)_{i,y}^-\Vert_{J_j}^2\right)^{1/2}\\
&\cdot\left(\Vert(\xi_u)_x\Vert_{[x_{N/2+1},1]\cup[0,1]}^2+\sum_{i=N/2+2}^{N-1}\sum_{j=1}^N h_{x,i+1}^{-1}\Vert[[\xi_u]]_{i,y}\Vert_{J_j}^2\right)^{1/2}+CN^{-(k+1)}|||\xi|||_E\\
\le&C\varepsilon^{-1/2}\left(\Vert\eta_u\Vert_{[x_{N/2+1},1]\cup[0,1]}^2+\sum_{i=N/2+2}^{N-1}\sum_{j=1}^N h_{x,i+1}\Vert(\eta_u)_{i,y}^-\Vert_{J_j}^2\right)^{1/2}(|||\xi|||_E+N^{-(k+1)})\\
&+CN^{-(k+1)}|||\xi|||_E\\
\le&CN^{-(k+1)}\ln^{1/2}N(|||\xi|||_E+N^{-(k+1)}).
\end{aligned}
\end{equation}

Substitute \eqref{B43-1}, \eqref{B43-2} and \eqref{B43-3,4} into \eqref{B43}, we can prove Lemma \ref{estimate-B43}.
\end{proof}

\begin{lemma}\label{estimate-B44}
\begin{equation*}
|\mathcal{B}_{44}^x(\eta_u;\xi_u)|\le CN^{-(k+1)}|||\xi|||_E.
\end{equation*}
\end{lemma}
\begin{proof}
Consider the novel interpolation \eqref{new interpolation} and the local Gauss-Radau projections $\pi^-$, $\pi_x^-$, one has
\begin{align*}
\mathcal{B}_{44}^x(\eta_u;\xi_u)=&-\sum_{j=1}^N(a_1)_{N/2,y}\left<(\eta_u)_{N/2,y}^-,(\xi_u)_{N/2,y}^+\right>_{J_j}\\
=&-(\sum_{j=1}^{N/2-1}+\sum_{j=N/2+1}^{N})(a_1)_{N/2,y}\left<(\eta_u)_{N/2,y}^-,(\xi_u)_{N/2,y}^+\right>_{J_j}\\ &-(a_1)_{N/2,y}\left<(\eta_u)_{N/2,y}^-,(\xi_u)_{N/2,y}^+\right>_{J_{N/2}}\\
=&-(a_1)_{N/2,y}\left<(\eta_u)_{N/2,y}^-,(\xi_u)_{N/2,y}^+\right>_{J_{N/2}}.
\end{align*}

To get an accurate estimate for $-(a_1)_{N/2,y}\left<(\eta_u)_{N/2,y}^-,(\xi_u)_{N/2,y}^+\right>_{J_{N/2}}$, we adopt a skillful technique by decomposing it into the following two terms:
\begin{align*}
&-(a_1)_{N/2,y}\left<(\eta_u)_{N/2,y}^-,(\xi_u)_{N/2,y}^+\right>_{J_{N/2}}\\
=&-(a_1)_{N/2,y}\left<(\eta_u)_{N/2,y}^-,[[\xi_u]]_{N/2,y}\right>_{J_{N/2}}-(a_1)_{N/2,y}\left<(\eta_u)_{N/2,y}^-,(\xi_u)_{N/2,y}^-\right>_{J_{N/2}},
\end{align*}
where from H\"{o}lder inequalities, Lemma \ref{interpolation error estimates} and \eqref{energy norm}, we can derive
\begin{equation*}
|-(a_1)_{N/2,y}\left<(\eta_u)_{N/2,y}^-,[[\xi_u]]_{N/2,y}\right>_{J_{N/2}}|\le CN^{-(k+3/2)}|||\xi|||_E,
\end{equation*}
and by H\"{o}lder inequalities, Lemma \ref{interpolation error estimates}, trace inequality and \eqref{energy norm}, one can obtain
\begin{equation*}
\begin{aligned}
	|-(a_1)_{N/2,y}\left<(\eta_u)_{N/2,y}^-,(\xi_u)_{N/2,y}^-\right>_{J_{N/2}}|\le &CN^{-(k+3/2)} N^{1/2}\Vert\xi_u\Vert_{K_{N/2,N/2}}\\
	\le& CN^{-(k+1)}|||\xi|||_E.
\end{aligned}
\end{equation*}

Thus we have proven Lemma \ref{estimate-B44}.
\end{proof}

At this point, we will present the main conclusion of this paper.
\begin{theorem}\label{theorem}
	Let $w = (u, p, q)=(u, \varepsilon u_x, \varepsilon u_y)$, where $u$ is the exact solution of problem \eqref{eq1.1}. Define the interpolation of $w$ as $\pi w=(P^{-}u,P_{x}^{+}p,P_{y}^{+}q)\in \mathcal{V}_N^3$, and the LDG solution as $W = (U, P, Q) \in \mathcal{V}_N^3$. Then we can establish the following supercloseness result:
	\begin{align*}
		|||\pi w-W|||_{E}\le  C N^{-(k+1)}\ln^{1/2}N.
	\end{align*}
\end{theorem}
\begin{proof}
According to \eqref{main-1}, Lemmas \ref{estimate-B43} and \ref{estimate-B44}, one has
\begin{equation*}
|\mathcal{B}_4^x(\eta_u;\xi_u)|\le CN^{-(k+1)}\ln^{1/2}N(|||\xi|||_E+N^{-(k+1)}).
\end{equation*}
We can estimate $\mathcal{B}_4^y(\eta_u;\xi_u)$ analogously and consequently, one obtains
\begin{equation}\label{main-2}
|\mathcal{B}_4(\eta_u;\xi_u)|\le CN^{-(k+1)}\ln^{1/2}N(|||\xi|||_E+N^{-(k+1)}).
\end{equation}

\eqref{main-2}, combined with \eqref{have been given-1} gives
\begin{equation*}
|||\xi|||_E^2=B(\eta;\xi)=\sum_{i=1}^{3}\mathcal{B}_i(\eta;\xi)+\mathcal{B}_4(\eta_u;\xi_u)\le CN^{-(k+1)}\ln^{1/2}N(|||\xi|||_E+N^{-(k+1)}),
\end{equation*}
i.e.
\begin{equation*}
|||\xi|||_E=|||\pi w-W|||_E\le CN^{-(k+1)}\ln^{1/2}N.
\end{equation*}

Thus we have done.
\end{proof}
\begin{remark}
With further derivation, we can obtain the $L^2$ error estimate:
\begin{equation*}
|||w-W|||_2\le CN^{-(k+1)}\ln^{1/2}N.
\end{equation*}	
\begin{proof}
Triangle inequality yields
\begin{equation*}
|||w-W|||_2=|||\eta-\xi|||_2\le|||\eta|||_2+|||\xi|||_2,
\end{equation*}
where from \eqref{energy norm} and Lemma \ref{interpolation error estimates}, one can get
\begin{equation*}
|||\eta|||_2\le C(\Vert\eta_u\Vert+\varepsilon^{-1/2}\Vert\eta_p\Vert+\varepsilon^{-1/2}\Vert\eta_q\Vert)\le CN^{-(k+1)},
\end{equation*} 
and from \eqref{energy norm} and Theorem \ref{theorem}, one has
\begin{equation*}
	|||\xi|||_2\le|||\xi|||_E\le CN^{-(k+1)}\ln^{1/2}N.
\end{equation*}
\end{proof}
\end{remark}

\begin{remark}
The primary emphasis of this article is on the level of numerical analysis, while the corresponding algorithm is consistent with the one employed in \cite{Cheng1Jiang2:2022-Supercloseness} for Bakhvalov-type mesh. Hence, we omit numerical experiments in this article, and recommend consulting \cite[Section 5]{Cheng1Jiang2:2022-Supercloseness} for relevant numerical experiments, which can also support our main conclusion, i.e., Theorem \ref{theorem}.
\end{remark}

\section{Conflict of interest statement}
We declare that we have no conflict of interest.

%
%%%%%%%%%%%%%%%%%%%%%%%%%%%%%%%%%%%%%%%%%%%%%%%%%%%%%%%%%%%%%%%%%%%%%%%%%%%%%%%%%%%%%%%%%%%%%%%%%%%%%%%%%%%%%%%

\bibliographystyle{plain}

%\bibliography{paper3-reference}
\begin{comment}
	
\end{comment}

%%%%%%%%%%%%%%%%%%%%%%%%%%%%%%%%%%%%%%%%%%%%%%%%%%%%%%%%%%%%%%%%%%%%%%%%%%%%%%%%%%%%%%%%%%%%%%%%%%%%%%%%%%%%%%%
%
%
%%%%%%%%%%%%%%%%%%%%%%%%%%%%%%%%%%%%%%%%%%%%%%%%%%%%%%%%%%%%%%%%%%%%%%%%%%%%%%%%%%%%%%%%%%%%%%%%%%%%%%%%%%%%%%%

%%%%%%%%%%%%%%%%%%%%%%%%%%%%%%%%%%%%%%%%%%%

\end{document}